\documentclass[10pt,leqno]{amsart}

\usepackage{amsmath}
\usepackage{amsfonts}
\usepackage{amssymb}
\usepackage{graphicx}
\usepackage{color}
\usepackage{hyperref}
\usepackage{enumerate}
\usepackage{xcolor}

\newtheorem{theorem}{Theorem}[section]
\newtheorem{question}{Question}[section]
\newtheorem*{theorem*}{Theorem}

\newtheorem{corollary}[theorem]{Corollary}
\newtheorem{proposition}[theorem]{Proposition}
\newtheorem{remark}[theorem]{Remark}

\def\Ric{\text{Ric}}

\def\l{\lambda}

\def\Ric{\operatorname{Ric}}
\def\Rm{\operatorname{Rm}}

\numberwithin{equation}{section}
\makeatletter
\newcommand*\owedge{\mathpalette\@owedge\relax}
\newcommand*\@owedge[1]{%
  \mathbin{%
    \ooalign{%
      $#1\m@th\bigcirc$\cr
      \hidewidth$#1\m@th\wedge$\hidewidth\cr
    }%
  }%
}
\makeatother

\begin{document}

\title[]{K\"ahlerity of Einstein four-manifolds}


\author{Xiaolong Li}\thanks{The first author's research is partially supported by a start-up grant at Wichita State University}
\address{Department of Mathematics, Statistics and Physics, Wichita State University, Wichita, KS, 67260}
\email{xiaolong.li@wichita.edu}

\author{Yongjia Zhang}
\address{School of Mathematical Sciences, Shanghai Jiao Tong University, Shanghai, China, 200240}
\email{zhangyongjia918@163.com}


\begin{abstract}
We prove that a closed oriented Einstein four-manifold is either anti-self-dual or (after passing to a double Riemannian cover if necessary) K\"ahler-Einstein, provided that $\l_2 \geq -\frac{S}{12}$, where $\l_2$ is the middle eigenvalue of the self-dual Weyl tensor $W^+$ and $S$ is the scalar curvature. 
An analogous result holds for closed oriented four-manifolds with $\delta W^+=0$. 
\end{abstract}

\maketitle

\section{Introduction}

We are concerned with the following question in this paper. 
\begin{question}\label{question}
Under what curvature conditions is a closed Einstein four-manifold K\"ahler? 
\end{question}

There have been several answers to this question. 
Derdzinski proved a fundamental result \cite{Derdzinski83} that if the self-dual Weyl tensor $W^+$ of a closed oriented Einstein four-manifold is parallel and  has at most two distinct eigenvalues at every point, then, after passing to a double Riemannian cover if necessary, the metric is K\"ahler.  
A result of Micallef and Wang \cite[Theorem 4.2]{MW93} asserts that a closed oriented four-manifold with harmonic self-dual Weyl tensor ($\delta W^+=0$) and half nonnegative isotropic curvature either is anti-self-dual or has a K\"ahler universal cover with constant positive scalar curvature (see also \cite{GL99}, \cite{RS16}, \cite{FKP14}, and \cite{Wu17TAMS} for alternative proofs).  

The purpose of this paper is to provide another curvature condition that characterizes closed K\"ahler-Einstein four-manifolds. To state our results, let us recall that the bundle $\Lambda^2$ of two-forms over an oriented (that is, orientable with an appointed orientation) Riemannian four-manifold $(M^4,g)$ admits an orthogonal decomposition $$\Lambda^2=\Lambda^+\oplus\Lambda^-,$$ 
with $\Lambda^+$ and $\Lambda^-$ being the $+1$ and $-1$ eigenspaces of the Hodge star operator, respectively. 
Sections of $\Lambda^+$ are called self-dual two-forms, while sections of $\Lambda^-$ are called anti-self-dual two-forms.  
The Riemann curvature tensor may be identified with a self-adjoint linear map 
$$\Rm:\Lambda^2\to \Lambda^2$$
called the curvature operator, via 
$$\Rm(e_i \wedge e_j)=\frac{1}{2} R_{ijkl}e_k \wedge e_l.$$
Corresponding to the decomposition $\Lambda^2=\Lambda^+\oplus\Lambda^-$, the Riemann curvature operator $\Rm$
can be decomposed into the following irreducible pieces:
\begin{align*}
    \Rm=\left[
\begin{array}{c|c}
 & \\
W^++\frac{S}{12}I & \overset{\circ}{\Ric}\\
 & \\
\hline
 & \\
 \overset{\circ}{\Ric} & W^-+\frac{S}{12}I \\
 & 
\end{array}
\right],
\end{align*}
where $S$ is the scalar curvature, $I$ is the identity map, $\overset{\circ}{\Ric}=\Ric -\tfrac{S}{4}g$ is the traceless Ricci tensor, and $W^{\pm}:\Lambda^{\pm} \to \Lambda^{\pm}$ are called the self-dual and anti-self-dual parts of the Weyl curvature operator $W:\Lambda^2\to\Lambda^2$. 
It is also convenient to denote by $\Rm^+$ and $\Rm^-$ the restriction of the Riemann curvature operator on $\Lambda^+$ and $\Lambda^-$, respectively. 
In other words, $\Rm^{\pm}=W^{\pm}+\frac{S}{12}I$. 
A four-manifold $(M^4,g)$ is said to be anti-self-dual if $W^+\equiv 0$ and self-dual if $W^-\equiv0$. 

It is well-known that on any K\"ahler surface $(M^4,J,g)$ with the natural orientation (in the sense that the K\"ahler form is self-dual), 
the self-dual Weyl operator $W^+:\Lambda^+ \to \Lambda^+$ is given by  
\begin{align*}
    W^+=\left[\begin{array}{ccc}
         \tfrac{S}{6}& &  \\
         & -\tfrac{S}{12} & \\
         & & -\tfrac{S}{12}
    \end{array}\right].
\end{align*}
In particular, if $\l_1 \leq \l_2 \leq \l_3$ are the eigenvalues of $W^+$, then we have
$$\l_2 = -\frac{S}{12}$$
on any K\"ahler surface, regardless of the sign of $S$. Equivalently, the middle eigenvalue of $\Rm^+$ vanishes on any K\"ahler surfaces. 

Our main result states that if a closed oriented Einstein four-manifold satisfies the condition 
\begin{equation*}
    \l_2 \geq -\frac{S}{12},
\end{equation*}
then it is either anti-self-dual or K\"ahler-Einstein (after passing to a double cover if necessary). 
More precisely, we prove
\begin{theorem}\label{theorem_main_0}
Let $(M^4,g)$ be a closed oriented Einstein four-manifold.
Denote by $\l_1\leq \l_2 \leq \l_3$ the eigenvalues of $W^+$ and by $S$ the scalar curvature.
If $\l_2 \geq -\frac{S}{12}$ everywhere, then one of the following assertions holds: 
\begin{enumerate}
    \item $S \neq 0$ and $(M,g)$ is K\"ahler-Einstein, or has a double Riemannian cover $\pi : (M^*,g^*)\to (M,g)$, where $g^*=\pi^*g$, such that $(M^*,g^*)$ is K\"ahler-Einstein.
    \item $(M,g)$ is anti-self-dual with nonnegative scalar curvature. Moreover, 
    \begin{itemize}
        \item if $S>0$, then $M$ is isometric to a quotient of $\mathbb{S}^4$ with the round metric or $\mathbb{CP}^2$ with the Fubini-Study metric.
        \item if $S=0$, then either $M$ is flat or its universal cover is a K3 surface with the Calabi-Yau metric.
    \end{itemize} 
\end{enumerate}
\end{theorem}

In the nonnegative scalar curvature case, Theorem \ref{theorem_main_0} provides a generalization of the result of Micallef and Wang \cite{MW93}, as the condition $\l_2 \geq -\frac{S}{12}$ is weaker than half nonnegative isotropic curvature. Recall that an oriented four-manifold is said to have half nonnegative (resp. positive) isotropic curvature if $\Rm^+$ is two-nonnegative (resp. two-positive), i.e., the sum of the smallest two eigenvalues of $\Rm^+$ is nonnegative (resp. positive). 
In view of $\Rm^+=W^+ +\frac{S}{12}I$, we see that half nonnegative isotropic curvature is equivalent to
\begin{equation*}
    S \geq 0 \quad\text{ and }\quad  \l_1+\l_2 \geq -\tfrac{S}{6},
\end{equation*}
so it is stronger than the condition $\l_2 \geq -\frac{S}{12}$.  

In the negative scalar curvature case, Theorem \ref{theorem_main_0}, to the best of our knowledge, seems to be the only answer to Question \ref{question} other than that of Derdzinski \cite{Derdzinski83}.

As a corollary of Theorem \ref{theorem_main_0}, we have a sufficient and necessary condition for a closed Einstein four-manifold to be K\"ahler.

\begin{corollary}\label{corollary_main_0}
Let $(M^4,g)$ be a closed and oriented (positive or negative) Einstein four-manifold which satisfies $H^1(M;\mathbb{Z}_2)=0$ and is not isometric to $\mathbb{S}^4$ with the round metric. Then $(M^4,g)$ is K\"ahler if and only if $\lambda_2(W^+) \ge -\frac{S}{12}$ everywhere. 
\end{corollary}

In addition, Theorem \ref{theorem_main_0} also implies some interesting classification results. 
In the positive scalar curvature case, we have
\begin{corollary}\label{corollary classification 1}
Let $(M^4,g)$ be a closed simply-connected Einstein four-manifold with positive scalar curvature. 
If $\l_2(W^+) > -\frac{S}{12}$ on $M$, then $W^+ \equiv 0$ and $M$ is isometric to $\mathbb{S}^4$ with the round metric or $\mathbb{CP}^2$ with the Fubini-Study metric. 
\end{corollary}
\begin{remark}
Corollary \ref{corollary classification 1} was proved by Polombo \cite{Polombo92} assuming the stronger condition of half positive isotropic curvature, which is equivalent to 
\begin{equation*}
    S > 0 \quad\text{ and } \quad \l_1+\l_2 > -\tfrac{S}{6}.
\end{equation*}
Alternative proofs of Polombo's result can be found in \cite{MW93}, \cite{FKP14}, \cite{RS16}, and \cite{Wu17TAMS}. 
\end{remark}
\begin{remark}
The assumption $\l_2 > -\frac{S}{12}$ in Corollary \ref{corollary classification 1} cannot be relaxed in view of the Einstein manifold $\mathbb{S}^2 \times \mathbb{S}^2$, which has $\l_2=-\frac{S}{12}$ everywhere. 
\end{remark}

In the Ricci-flat case, we have 
\begin{corollary}\label{corollary classification 2}
Let $(M^4,g)$ be a closed oriented Ricci-flat four-manifold.
If $\l_2(W^+) \geq 0$ on $M$, then $W^+ \equiv 0$, and either $M$ is flat or its universal cover is a K3 surface with the Calabi-Yau metric.
\end{corollary}
\begin{remark}
Corollary \ref{corollary classification 2} was proved by Micallef and Wang \cite{MW93} under the stronger condition of half-nonnegative isotropic curvature. Alternative proofs of their result can be found in \cite{FKP14}, \cite{RS16}, and \cite{Wu17TAMS}. 
However, it seems that none of these proofs work under our weaker curvature assumption.  
\end{remark}

With a bit more efforts, the method employed to prove Theorem \ref{theorem_main_0} also yields analogous results for closed four-manifolds with harmonic (namely, divergence free) self-dual Weyl tensor, i.e. 
\begin{equation}\label{divergence_free_Weyl}
    \delta W^+:=-\nabla\cdot W^+=0.
\end{equation}
All Einstein metrics satisfy \eqref{divergence_free_Weyl} as a consequence of the second Bianchi identity. However, \eqref{divergence_free_Weyl} is actually much weaker than the Einstein condition. 

\begin{theorem}\label{Theorem_main_1}
Let $(M^4,g)$ be a closed oriented four-manifold with $\delta W^+=0$. 
Denote by $\l_1\leq \l_2 \leq \l_3$ the eigenvalues of $W^+$ and by $S$ the scalar curvature.
If $\l_2 \geq -\frac{S}{12}$ everywhere, then one of the following assertions holds: 
\begin{enumerate}[(1)]
    \item $(M^4,g)$ has constant nonzero scalar curvature, and $(M^4,g)$ is K\"ahler, or has a double Riemannian cover $\pi : (M^*,g^*)\to (M,g)$, where $g^*=\pi^*g$, such that $(M^*,g^*)$ is K\"ahler. 
    \item $(M^4,g)$ is anti-self-dual with nonnegative scalar curvature.
\end{enumerate}
\end{theorem}

We would like to point out that both Theorem \ref{theorem_main_0} and Theorem \ref{Theorem_main_1} are also motivated by an important work of Peng Wu \cite{Wu21}, in which he found a beautiful characterization of conformally K\"ahler Einstein metrics of positive scalar curvature on closed oriented four-manifolds via the condition $\det(W^+) >0$. 
LeBrun \cite{Lebrun21}, based on his earlier work in \cite{LeBrun15}, provided an entirely different proof of the result and relaxed the assumptions to $W^+ \neq 0$ and $|W^{+}|^{-3} \det(W^+) \geq -\tfrac{5\sqrt{2}}{21\sqrt{21}}$. 
In comparison with Wu and Lebrun's result, Theorem \ref{theorem_main_0} has the advantage of giving a curvature characterization in the negative scalar curvature case, but has the disadvantage of using a closed curvature condition. Moreover, the condition $\l_2 \geq -\frac{S}{12}$ is not conformally invariant, thus it does not allow any conformally K\"ahler manifolds.

Theorem \ref{Theorem_main_1} implies the following results, which are known under the stronger assumption of half positive (resp. nonnegative) isotropic curvature (see \cite{Polombo92}, \cite{MW93}, \cite{FKP14}, \cite{RS16}, and \cite{Wu17TAMS}). 
\begin{corollary}\label{corollary classification 3}
Let $(M^4,g)$ be a closed oriented four-manifold with $\delta W^+=0$ and positive scalar curvature. 
If $\l_2 > -\frac{S}{12}$ on $M$, then $W^+ \equiv 0$. 
\end{corollary}
\begin{corollary}\label{corollary classification 4}
Let $(M^4,g)$ be a closed oriented four-manifold with $\delta W^+=0$ and nonnonegative scalar curvature.
If $\l_2 \geq 0$ on $M$, then $W^+ \equiv 0$.
\end{corollary}

To conclude this introduction, we give a brief discussion of our strategies to prove the above-mentioned results. 
In order to prove Theorem \ref{theorem_main_0}, we apply the maximum principle to the following partial differential inequality
\begin{equation}\label{eq 1.2}
    \Delta(\l_3-\l_1) \geq 6(\l_3 -\l_1)(\l_2+\tfrac{S}{12})
\end{equation}
to conclude that $\l_3-\l_1$ must be a constant. Moreover, we must have $\l_2=-\frac{S}{12}$ everywhere unless $M$ is anti-self-dual. 
This together with the Einstein condition implies that all the $\l_i$'s are constant functions, and their values can be read from the differential inequalities satisfied by them. The desired K\"ahlerity then follows from the work of Derdzinski \cite{Derdzinski83}. 

The key to prove Theorem \ref{Theorem_main_1} is to show that the scalar curvature must be constant unless $M$ is anti-self-dual. We achieve this by picking up some extra gradient terms in \eqref{eq 1.2}, which are used to conclude that the $\l_i$'s are constants. It is worth mentioning that these helpful extra terms are used in Wu \cite{Wu21} as well.

The proofs of Theorem \ref{theorem_main_0} and Theorem \ref{Theorem_main_1} are given in Section 2 and Section 3, respectively.

\section{The Einstein Case}
In this section, we prove Theorem \ref{theorem_main_0}. First of all, we prove two technical propositions which will also be applied in Section 3. Note that Proposition \ref{prop 2.1} assumes only harmonic self-dual Weyl tensor, while Proposition \ref{prop 2.2} assumes, in addition, constant scalar curvature; these assumptions are consequences of the Einstein condition.

\begin{proposition}\label{prop 2.1}
Let $(M^4,g)$ be a closed oriented four-manifold with $\delta W^+=0$. Denote by $\l_1\leq \l_2 \leq \l_3$ the eigenvalues of $W^+$. If $\l_2 \geq -\frac{S}{12}$ everywhere, then either  $W^+ \equiv 0$ or $\l_3-\l_1$ is equal to a positive constant and $\l_2 =-\frac{S}{12}$ everywhere. 
\end{proposition}

\begin{proposition}\label{prop 2.2}
Let $(M^4,g)$ be a closed oriented four-manifold with $\delta W^+=0$ and with constant scalar curvature $S$. Denote by $\l_1\leq \l_2 \leq \l_3$ the eigenvalues of $W^+$. If $\l_2 \geq -\frac{S}{12}$ everywhere and $W^+ \not\equiv 0$, then the following statements hold:
\begin{enumerate}
    \item $\l_2 = -\frac{S}{12}$;
    \item $\l_1 = -\frac{S}{12}$ and $\l_3=\frac{S}{6}$ if $S>0$, and $\l_1 = \frac{S}{6}$ and $\l_3=-\frac{S}{12}$ if $S<0$;
    \item $\nabla W^+ =0$. 
\end{enumerate}
\end{proposition}

\begin{proof}[Proof of Proposition \ref{prop 2.1}]
According to \cite[page 664]{MW93}, $\delta W^+=0$ is equivalent to the Weitzenb\"ock formula
\begin{equation}\label{Weitzenbock1}
    \Delta W^+ =\tfrac{S}{2}W^+ -2(W^+)^2 -4(W^+)^{\#},
\end{equation}
where the $(W^+)^\#$ is the adjoint matrix of $W^+$.
The reader may consult \cite{hamilton86} or \cite{CNL06} for more information.   
If we choose an orthonormal basis $\{\omega_{i}\}_{i=1}^3$ of $\Lambda^+$ that diagonalizes $W^+$ with eigenvalues $\l_1 \leq \l_2 \leq \l_3$, then, with respect to this basis, we have
\begin{align*}
    (W^+)^2=\left[\begin{array}{ccc}
        \l_1^2 &  &  \\
         & \l_2^2  &  \\
         & & \l_3^2
    \end{array}\right]
\end{align*}
and 
\begin{align*}
    (W^+)^\#=\left[\begin{array}{ccc}
        \l_2 \l_3 &  &  \\
         & \l_1 \l_2  &  \\
         & & \l_1\l_2
    \end{array}\right].
\end{align*}
Therefore, the Weitzenb\"ock formula \eqref{Weitzenbock1} implies the following differential inequalities for the eigenvalues $\lambda_1$ and $\lambda_3$
\begin{eqnarray}
    \Delta \l_1 &\leq& \tfrac{S}{2}\l_1 -2\l_1^2 -4\l_2 \l_3,\label{simplepdel1} \\
    \Delta \l_3 &\geq& \tfrac{S}{2}\l_3 -2\l_3^2 -4\l_1 \l_2. \label{simplepdel3}
\end{eqnarray}
Since $\l_1$ and $\l_3$ are only locally Lipschitz functions on $M$, the inequalities above are all understood in the sense of barrier (see \cite{Calabi58}). To prove \eqref{simplepdel1} and \eqref{simplepdel3}, one may construct the barriers as follows. Fix $p\in M$ and let $\omega\in \Lambda^+_p$ be the unit eigenvector of $W^+_p$ associated with $\lambda_3$. Now we may extend $\omega$ to a neighborhood of $p$ by parallel transport and compute using the normal coordinates at $p$:
\begin{align*}
    \Delta(W^+(\omega,\omega))&= (\Delta W^+)(\omega,\omega)
    \\
    & =\tfrac{S}{2}W^+(\omega,\omega)-2|W^+(\omega)|^2-4(W^+)^\#(\omega,\omega)
    \\
    &=\tfrac{S}{2}\l_3-2\l_3^2-4\l_1\l_2 \quad\text{ at } p.
\end{align*}
Obviously, $W^+(\omega,\omega)$ is a lower barrier of $\l_3$ at $p$, and this proves \eqref{simplepdel3}; the proof of \eqref{simplepdel1} is similar.

Subtracting \eqref{simplepdel1} from \eqref{simplepdel3} yields
\begin{eqnarray*}
    \Delta (\l_3-\l_1) &\geq& \tfrac{S}{2}(\l_3-\l_1) -2(\l_3^2 -\l_1^2) +4\l_2 (\l_3 -\l_1)  \\
    &=& (\l_3-\l_1)\left(\tfrac{S}{2}-2\l_1 -2\l_3 +4\l_2  \right) \\
    &=& 6(\l_3-\l_1)\left(\l_2 +\tfrac{S}{12}\right),
\end{eqnarray*}
where we have used $\l_1+\l_2+\l_3=0$ in the last step. 
It follows from $\l_2 \geq -\frac{S}{12}$ that 
\begin{equation*}
    \Delta (\l_3-\l_1) \geq 6(\l_3-\l_1)\left(\l_2 +\tfrac{S}{12}\right) \geq 0,
\end{equation*}
in the sense of barrier. 
By the maximum principle, we conclude that $\l_3-\l_1 \equiv c \geq 0$.  
Note that $c=0$ implies $W^+\equiv 0$. 
In case $c>0$, the desired equality $\l_2 =-\frac{S}{12}$ follows from 
\begin{equation*}
    0=\Delta (\l_3-\l_1) \geq 6(\l_3-\l_1)\left(\l_2 +\tfrac{S}{12}\right).
\end{equation*}
The proof is finished.
\end{proof}

\begin{proof}[Proof of Proposition \ref{prop 2.2}]
Part (1) follows immediately from Proposition \ref{prop 2.1} since $W^+ \not\equiv 0$. 

To prove part (2), we first notice that $\l_3 -\l_1$ is equal to a positive constant by Proposition \ref{prop 2.1}. 
Now the constant scalar curvature assumption implies that $\l_1 +\l_3$ is a constant function, as 
$$\l_1 +\l_3=-\l_2=\tfrac{S}{12}.$$
It then follows that both $\l_1$ and $\l_3$ must be constant functions. 
Substituting $\l_2 = -\frac{S}{12}$ and $\l_1+\l_2+\l_3 =0$ into the differential inequalities satisfied by $\l_1$ and $\l_3$, we obtain that
\begin{eqnarray*}
    0 = \Delta \l_1 &\leq& \tfrac{S}{2}\l_1 -2\l_1^2 -4\l_2 \l_3 =-2(\l_1+\tfrac{S}{12})(\l_1 -\tfrac{S}{6})\\
    0 = \Delta \l_3 &\geq& \tfrac{S}{2}\l_3 -2\l_3^2 -4\l_1 \l_2 =-2(\l_3+\tfrac{S}{12})(\l_3 -\tfrac{S}{6}). 
\end{eqnarray*}
One easily reads from above inequalities that we must have 
\begin{eqnarray*}
    \l_1 =\l_3=0, && \text{ if } S=0; \\
     \l_1=-\tfrac{S}{12} \text{ and } \l_3=\tfrac{S}{6}, && \text{ if } S>0; \\
    \l_1=\tfrac{S}{6} \text{ and } \l_3=-\tfrac{S}{12}, && \text{ if } S<0. 
\end{eqnarray*}
Now part (2) is proved. 

Part (3) follows immediately from part (2) and the constant scalar curvature assumption. 

\end{proof}

We now give the proof of Theorem \ref{theorem_main_0}. 
\begin{proof}[Proof of Theorem \ref{theorem_main_0}]

If $M$ is anti-self-dual, then $0=\lambda_2\ge-S/12$ implies that $S$ is a nonnegative constant. By Hitchin's classification \cite{Hitchin74} of half conformally flat Einstein four-manifolds with nonnegative scalar curvature (see also \cite[Theorem 13.30]{Besse08}), $(M^4,g)$ is one of the manifolds as described in part (2) of the statement.

If $M$ is not anti-self-dual, then by Proposition \ref{prop 2.2}, we have that $\nabla W^+ =0$ and $W^+$ has at most two distinct eigenvalues at every point. 
One can then invoke the result of Derdzinski \cite[Therorem 2]{Derdzinski83} to conclude that either $M$ is K\"ahler or it has a double cover that is K\"ahler.
Below we provide a more direct proof using an elegant argument of Lebrun \cite{Lebrun21}, which is also applicable to the proof of Theorem \ref{Theorem_main_1}. 
We shall only present the case where $M$ has negative constant scalar curvature, as it can be easily adapted to the positive scalar curvature case by flipping signs and reversing directions of the related inequalities. 

Since $\l_1=S/6$ is an isolated eigenvalue of $W^+$, we have that the corresponding eigenspaces of $W^+$ at each point on $M$ form a one-dimensional subbundle of $\Lambda^+$. Denote this bundle by $L$. If $L$ is a trivial bundle (in particular, if $H^1(M;\mathbb{Z}_2)=0$), then we can find a non-vanishing global section $\omega$ of $L$ with constant norm everywhere. If $L$ does not have a non-vanishing global section, then we may let $M^*$ be all the elements in $L$ with constant unit norm, and the restriction of the bundle projection $\pi:M^*\to M$ is a double cover. If we equip $M^*$ with the Remannian metric $g^*=\pi^*g$, then $(M^*,g^*)$ also satisfies Proposition \ref{prop 2.1} and $W^+$ admits a global eigenvector section associated with $\lambda_1=S/6$. Henceforth, we will assume that $L$ admits a global section $\omega$ with constant norm $|\omega|\equiv \sqrt{2}$. 

Since $\omega$, when viewed as an operator on $TM$, satisfies
\begin{equation*}
    \omega\cdot \omega = -\operatorname{id},
\end{equation*}
we have that $\omega$ provides an almost complex structure. We need only to show that $\omega$ is a closed form.

We first observe that 
\begin{equation}\label{eq 2.1}
    W^+(\nabla_X\omega,\nabla_X\omega)\geq 0.
\end{equation}
%
for any tangent vector $X$.
This is because $\omega$ has constant norm, which implies that $\langle\nabla_X\omega,\omega\rangle=0$ and hence $\nabla_X\omega$ is in the eigenspace of $\lambda_2=\lambda_3=-S/12>0$. 

Now recall that when $\delta W^+=0$, the following Weitzenb\"ock formula holds (c.f. \cite[(4.1c)]{MW93} or \cite[Equation (9)]{Lebrun21})
\begin{align}\label{Weitzenbock}
    \Delta W^+=\tfrac{S}{2}W^+-6W^+ \circ W^++2|W^+|^2I.
\end{align}
It is easy to check that \eqref{Weitzenbock} is equivalent to \eqref{Weitzenbock1}. We now take inner product of \eqref{Weitzenbock} and $\omega \otimes \omega$, and integrate over $M$ to get
\begin{align*}
    0&=\int_M\left\langle-\Delta W^++\tfrac{S}{2}W^+-6W^+ \circ W^++2|W^+|^2I,\omega\otimes\omega\right\rangle dg
    \\
    &=\int_M\Big(-2W^+(\omega,\Delta\omega)-2W^+(\nabla_e\omega,\nabla^e\omega)
    \\
    &\qquad +\tfrac{S}{2}\lambda_1|\omega|^2-6\lambda_1^2|\omega|^2+2(\lambda_1^2+\lambda_2^2+\lambda_3^2)|\omega|^2\Big)dg
    \\
    &\le -2\int_M\lambda_1\langle\omega,\Delta\omega\rangle dg
    \\
    &=\tfrac{S}{3}\int_M|\nabla \omega|^2dg,
\end{align*}
where we have applied Proposition \ref{prop 2.2}(1)(2) and \eqref{eq 2.1}. Since $S$ is a negative constant, we have that $\nabla \omega$ vanishes everywhere. It follows that $\omega$ is a K\"ahler form.

In case where $W^+\not\equiv0$ and $S$ is a positive constant, we can also find a K\"ahler form in the same way.

\end{proof}




It is clear that Corollary \ref{corollary_main_0},  Corollary \ref{corollary classification 1}, and  Corollary \ref{corollary classification 2} can be easily observed from the proof of Theorem \ref{theorem_main_0}.



\section{The Harmonic self-dual Weyl case}

If $M$ is only assumed to have harmonic self-dual Weyl tensor, then the proof given in the previous section breaks down because we cannot conclude that both $\l_1$ and $\l_3$ are constant functions without the constant scalar curvature assumption. 
We shall show, by improving the partial differential inequality satisfied by $\l_3-\l_1$, that the scalar curvature must be a constant unless $W^+ \equiv 0$, as a consequence of $\delta W^+=0$ and $\l_2 \geq -\frac{S}{12}$. 

\begin{proposition}\label{prop 3.1}
Let $(M^4,g)$ be a closed oriented four-manifold with $\delta W^+=0$. Denote by $\l_1\leq \l_2 \leq \l_3$ the eigenvalues of $W^+$. If $\l_2 \geq -\frac{S}{12}$ everywhere, then either $M$ is anti-self-dual with nonnegative scalar curvature, or $M$ has constant scalar curvature. 
\end{proposition}

\begin{proof}

Let us recall some computations in \cite{Derdzinski83}.
For any $x\in M$, we can choose an  orthogonal basis $\omega_1$, $\omega_2$, $\omega_3$ of $\Lambda^+_x$, consisting of eigenvectors of $W^+$ such that
\begin{equation}\label{eq 3.1}
    |\omega_1|^2=|\omega_2|^2=|\omega_3|^2=2.
\end{equation}
Consequently, we have that, at $x$
\begin{equation}\label{eq 3.2}
    W^+=\frac{1}{2}\left(\l_1 \omega_1 \otimes \omega_1 + \l_2 \omega_2 \otimes \omega_2 + \l_3 \omega_3 \otimes \omega_3 \right)
\end{equation}
with $\l_1 \leq \l_2 \leq \l_3$ being the eigenvalues of $W^+$. 
Let $M_{W^+}\subset M$ be the open dense set where the number of distinct eigenvalues of $W^+$ is locally constant. In $M_{W^+}$, the pointwise formula \eqref{eq 3.2} is valid locally in the sense that the mutually orthogonal sections $\omega_1, \omega_2, \omega_3$ of $\Lambda^+$ satisfying \eqref{eq 3.1} and the functions $\l_1, \l_2,\l_3$ may be assumed differentiable in a neighborhood of any $p\in M_{W^+}$. 


Since $\Lambda^+$ is invariant under parallel transport, in a neighborhood of $p \in M_{W^+}$, there exist one-forms $a,b$, and $c$ defined near $p$, such that we have \eqref{eq 3.2} and 
\begin{align*}
    \nabla \omega_1&=a\otimes\omega_2-c\otimes\omega_3,
    \\
    \nabla \omega_2&=b\otimes\omega_3-a\otimes\omega_1,
    \\
    \nabla \omega_3&=c\otimes\omega_1-b\otimes\omega_2.
\end{align*}
It was shown by Derdzinski \cite{Derdzinski83} that if $\delta W^+=0$, then in a neighborhood of $p \in M_{W^+}$, we have
\begin{eqnarray*}
    \nabla \lambda_1&=&(\lambda_2-\lambda_1)(\iota_{a^\#}\omega_3)^{\#}+(\lambda_3-\lambda_1)(\iota_{c^\#}\omega_2)^{\#},\label{grad1}
    \\
    \nabla \lambda_2&=&(\lambda_1-\lambda_2)(\iota_{a^\#}\omega_3)^{\#}+(\lambda_3-\lambda_2)(\iota_{b^\#}\omega_1)^{\#},\label{grad2}
    \\
    \nabla \lambda_3&=&(\lambda_1-\lambda_3)(\iota_{c^\#}\omega_2)^{\#}+(\lambda_2-\lambda_3)(\iota_{b^\#}\omega_1)^{\#},\label{grad3}
    \end{eqnarray*}
and 
\begin{eqnarray*}
    \Delta\lambda_1 &=& 2(\lambda_1-\lambda_2)|(\iota_{a^\#}\omega_3)^{\#}|^2+2(\lambda_1-\lambda_3)|(\iota_{c^\#}\omega_2)^{\#}|^2 \\
    && +\tfrac{S}{2}\lambda_1-2\lambda_1^2-4\lambda_2\lambda_3,
    \\
    \Delta\lambda_2&=&2(\lambda_2-\lambda_1)|(\iota_{a^\#}\omega_3)^{\#}|^2+2(\lambda_2-\lambda_3)|(\iota_{b^\#}\omega_1)^{\#}|^2 \\
    && +\tfrac{S}{2}\lambda_2-2\lambda_2^2-4\lambda_1\lambda_3,
    \\
    \Delta\lambda_3&=&2(\lambda_3-\lambda_1)|(\iota_{c^\#}\omega_2)^{\#}|^2+2(\lambda_3-\lambda_2)|(\iota_{b^\#}\omega_1)^{\#}|^2\\
    && +\tfrac{S}{2}\lambda_3-2\lambda_3^2-4\lambda_1\lambda_2,
\end{eqnarray*}
where $\iota$ is the interior product and $\#$ is the sharp operator. 

It follows from the assumption $\l_2\ge-S/12$ that in the set $M_{W^+}$, we have 
\begin{eqnarray*}
 \Delta(\lambda_3-\lambda_1) 
&= & 6(\l_3-\l_1)\left( \l_2+\tfrac{S}{12}\right) + 4(\lambda_3-\lambda_1)|(\iota_{c^\#}\omega_2)^{\#}|^2 \\
&& +2(\lambda_3-\lambda_2)|(\iota_{b^\#}\omega_1)^{\#}|^2  +2(\lambda_2-\lambda_1)|(\iota_{a^\#}\omega_3)^{\#}|^2 \\
&\geq &  4(\lambda_3-\lambda_1)|(\iota_{c^\#}\omega_2)^{\#}|^2  +2(\lambda_3-\lambda_2)|(\iota_{b^\#}\omega_1)^{\#}|^2 \\
&&+2(\lambda_2-\lambda_1)|(\iota_{a^\#}\omega_3)^{\#}|^2 .
\end{eqnarray*}
Since $\l_3-\l_1$ is a nonnegative constant on $M$ by Proposition \ref{prop 2.1}, 
we conclude that 
\begin{equation*}
    (\lambda_3-\lambda_1)|(\iota_{c^\#}\omega_2)^{\#}|^2=(\lambda_3-\lambda_2)|(\iota_{b^\#}\omega_1)^{\#}|^2=(\lambda_2-\lambda_1)|(\iota_{a^\#}\omega_3)^{\#}|^2=0,
\end{equation*}
in a neighborhood of $p\in M_{W^+}$. 
This implies that 
\begin{equation*}
    \nabla \l_1 =\nabla \l_3 =0
\end{equation*}
in that neighborhood of $p\in M_{W^+}$. Therefore, $\l_1$ and $\l_3$ are locally constant on $M_{W^+}$. Since $M_{W^+}$ is open and dense in $M$, and since $\l_1$ and $\l_3$ are locally Lipschitz functions, we conclude that $\l_1$ and $\l_3$ are global constant functions. Moreover $\l_2=-\l_1-\l_3$ is also a constant. 

If $M$ is not anti-self-dual, then we have $\l_2=-\frac{S}{12}$ by Proposition \ref{prop 2.1}. 
$S$ must be a constant since $\l_2$ is so. 

\end{proof}




\begin{proof}[Proof of Theorem \ref{Theorem_main_1}]
If $M$ is anti-self-dual, then we have $S\geq 0$ in view of $0=\l_2 \geq -\frac{S}{12}$. 
If $M$ is not anti-self-dual, then $M$ has constant scalar curvature by Proposition \ref{prop 3.1}. In the latter case, Proposition \ref{prop 2.2} is valid and the proof of K\"ahlerity is the same as the proof of Theorem \ref{theorem_main_0} given in Section 2. 
\end{proof}

It is clear that Corollary \ref{corollary classification 3} and Corollary \ref{corollary classification 4} can be observed from the proof of Theorem \ref{Theorem_main_1}.

\section*{Acknowledgment}
The authors would like to thank Professor Jiaping Wang for some helpful discussions related to this work.


\bibliographystyle{alpha}
\bibliography{ref}

\end{document}